\documentclass[12pt, reqno]{amsart}

\usepackage{amsmath}
\usepackage{amssymb}
\usepackage{fullpage}
\usepackage{amsthm}
\usepackage{graphicx}
\usepackage{placeins}
\usepackage{caption}
\usepackage{extpfeil}
\usepackage{enumerate}
\usepackage{caption}
\usepackage{subcaption}
\usepackage{hyperref}
\usepackage{xcolor}
\usepackage{pbox}

\definecolor{blue}{rgb}{0,0,0.9}
\definecolor{purple}{rgb}{0.6,0,0.9}

\newcommand{\be}{\begin{equation}}
\newcommand{\ee}{\end{equation}}
\newcommand{\bee}{\begin{equation*}}
\newcommand{\eee}{\end{equation*}}

\newtheorem{thm}{Theorem}
\newtheorem{prop}[thm]{Proposition}

\numberwithin{equation}{section}

\begin{document}

\title{Symmetry of hypersurfaces with ordered mean curvature in one direction}
\author{Yanyan Li$^1$, Xukai Yan$^2$, and Yao Yao$^3$}
\date{\today}
\thanks{$^1$Department of Mathematics, Rutgers University, 110 Frelinghuysen Road, Piscataway, NJ 08854. Email: yyli@math.rutgers.edu. YYL is partially supported by 
NSF Grants DMS-1501004, DMS-2000261, and Simons Fellows Award 677077}
\thanks{$^2$Department of Mathematics, Oklahoma State University, 401 Mathematical Sciences Building, Stillwater, OK 74078 USA. E-mail: xuyan@okstate.edu. 
 XY is partially supported by AMS-Simons Travel Grant and AWM-NSF Travel Grant 1642548.}
\thanks{$^3$School of Mathematics, Georgia Institute of Technology, 686 Cherry Street, Atlanta, GA 30332-0160 USA. E-mail: yaoyao@math.gatech.edu. YY is partially supported by NSF grants DMS-1715418 and DMS-1846745, and Sloan Research Fellowship.}

\begin{abstract}

For a connected $n$-dimensional compact smooth hypersurface $M$ without boundary embedded in $\mathbb{R}^{n+1}$, a classical result of Aleksandrov  shows that it must be a sphere if it has constant mean curvature. Li and Nirenberg studied a one-directional analog of this result: if every pair of points $(x',a), (x',b)\in M$ with $a<b$ has ordered mean curvature $H(x',b)\leq H(x',a)$, then $M$ is symmetric about some hyperplane $x_{n+1}=c$ under some additional conditions. Their proof was done by the moving plane method and some variations of the Hopf Lemma. We obtain the symmetry of $M$ under some weaker assumptions using a variational argument, giving a positive answer to the conjecture in \cite{LN2}.

\end{abstract}

\maketitle

\section{Introduction}

Let $M$ be a compact connected $C^2$ hypersurface without boundary embedded in $\mathbb{R}^{n+1}$. For $x\in M$, we denote its mean curvature by
$
H(x) = \frac{1}{n}\sum_{i=1}^n k_i(x), 
$ 
where $k_1(x),\dots,k_n(x)$ are the principal curvatures of $M$ at $x$ with respect to the outer normal. 

It is a classical problem to study how the symmetry of a hypersurface $M$ in $\mathbb{R}^{n+1}$ is related to its mean curvature, see e.g. Jellett \cite{Jellett}, Liebmann \cite{Liebmann} and Chern \cite{Chern}. 
Hopf \cite{Hopf} established that  
an immersion of a topological $2$-sphere in $\mathbb{R}^3$ with constant mean curvature must be a standard sphere, 
and raised the conjecture that the conclusion holds for all immersed connected closed hypersurfaces in $\mathbb{R}^{n+1}$ with constant mean curvature. 
Aleksandrov \cite{Aleksandrov} proved that if $M$ is an embedded connected closed hypersurface with constant mean curvature, then $M$ must be a standard sphere. 
If $M$ is immersed instead of embedded, then the conclusion does not hold in general. 
In dimensions $n\ge 3$, 
Hsiang \cite{Hsiang} showed the existence of immersions of $\mathbb{S}^n$ into $\mathbb{R}^{n+1}$ with constant mean curvatures but not  standard spheres. 
For $n=2$, Wente \cite{Wente} constructed immersions of  $2$-dimensional tori into $\mathbb{R}^3$ with constant mean curvatures. 
Kapouleas \cite{Kap1, Kap2} showed the existence of  closed two surfaces of genus $g$ immersed in $\mathbb{R}^3$ with constant mean curvatures, for every $g\ge 2$.
The same problem was also studied for $\sigma_m$-curvatures for $2\le m\le n$. For every $1\le m\le n$, the $\sigma_m$-curvature is the $m$-th elementary symmetric function of the principle curvatures, i.e.  $\sigma_m(x)=\Sigma_{1\le i_1<...<i_m\le n}k_{i_1}(x)\cdots k_{i_m}(x)$. (In particular,  $\sigma_1$-curvature corresponds to the mean curvature.)  Ros \cite{Ros2, Ros} proved that for any $2\le m\le n$, if $M$ is a closed connected hypersurface embedded in $\mathbb{R}^{n+1}$ with constant $\sigma_m$-curvature, then it must be a standard sphere. 

In this paper, we study a one-directional analog related to Aleksandrov's result \cite{Aleksandrov}. Given a special direction, e.g. the \emph{vertical} direction parallel to the $x_{n+1}$ axis,  we aim to answer the following question: What assumption on the mean curvature would guarantee the symmetry of $M$ about some hyperplane $x_{n+1}=c$? Although $M$ having constant mean curvature is sufficient, this assumption is clearly too strong. It would be more reasonable to impose some one-directional assumptions, such as the mean curvature being constant along each vertical line, or an even weaker assumption that the mean curvature is ordered along each vertical line.

In \cite{Li97}, Li proved that if the mean curvature $H:M\to\mathbb{R}$ has a 
 $C^1$ 
 extension $K:\mathbb{R}^{n+1}\to\mathbb{R}$ where $K$ has a non-positive partial derivative in the $x_{n+1}$ direction, then $M$ is symmetric about some hyperplane $x_{n+1}=c$.  Li then proposed to replace the above assumption by the following weaker and more natural assumption:


\noindent\textbf{Main Assumption.} Let $x'=(x_1,...,x_n)$.  Denote by $G$ the bounded open set in $\mathbb{R}^{n+1}$ bounded by the hypersurface $M$. 
For any two points $(x', a), (x', b)\in M$ satisfying $a<b$ and that $\{(x', \theta a+(1-\theta)b) : 0\leq \theta\leq 1\}$ lies in $G$, we have \begin{equation}\label{ineq}
H(x', b)\leq H(x', a).
\end{equation}


Li and Nirenberg showed  in \cite{LN1} that this assumption alone is not enough to guarantee the symmetry of $M$ about some hyperplane $x_{n+1}=c$. 
They also constructed a counterexample  \cite[Section 6]{LN1} where the inequality \eqref{ineq} does not imply a pairwise equality, and pointed out that even if  \eqref{ineq} is replaced by an equality, it still does not guarantee the symmetry of $M$, due to the counterexample in Figure~\ref{fig0}. 

In \cite{LN2}, they conjectured that the Main Assumption together with the following Condition S should imply the symmetry.

\begin{figure}[h!]
\hspace*{-0cm}\includegraphics[scale=0.8]{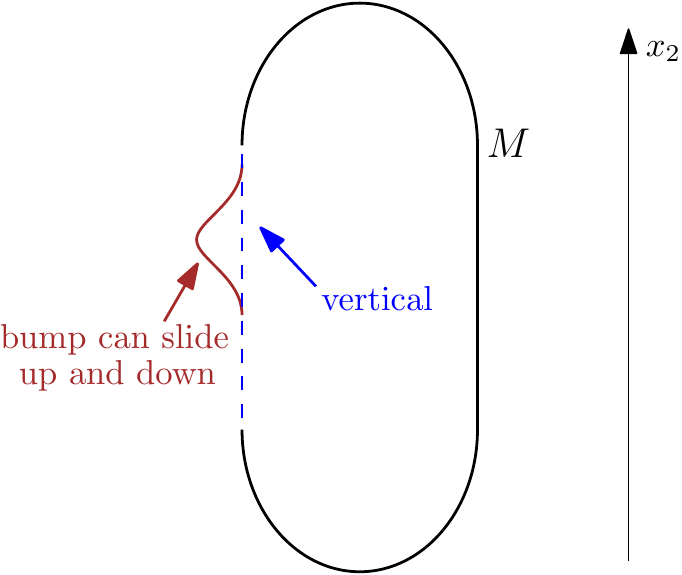}
\caption{Illustration of a smooth  curve in $\mathbb{R}^2$ that satisfies the Main Assumption with \eqref{ineq} being an equality for every pair of points, but it is not symmetric about any horizontal line. Note that it does not satisfy Condition~S or S'.\label{fig0}}
\end{figure}

\noindent\textbf{Condition S. } $M$ stays on one side of any hyperplane parallel to the $x_{n+1}$ axis that is tangent to $M$.

Note that Condition S holds for all convex $M$, but it does not require $M$ being convex.  In the case when $n=1$, when $M$ is a closed $C^2$ embedded curve in the plane satisfying both conditions above, \cite[Theorem 1.4]{LN1} proved the symmetry of $M$. In higher dimensions, Li and Nirenberg \cite[Theorem 1]{LN2} established the symmetry of $M$ under the following two assumptions, instead of Condition S: (1) Every line parallel to the $x_{n+1}$-axis that is tangent to $M$ has contact of finite order (note that every analytic $M$ satisfies this property); (2)  For every point on $M$ with a horizontal tangent, if $M$ is viewed locally as the graph of a function defined on the tangent plane, the function is locally concave near the contact point with respect to the outer normal. Note that neither of Condition S or (1)+(2) implies the other. Their proof is done by the moving plane method and some variations of the Hopf Lemma, and their result can also be extended to more general curvature functions other than the mean curvature.

In this paper, our goal is to prove the symmetry of $M$ under Condition S, which gives a positive answer to the conjecture in \cite{LN2}. In fact, we will replace Condition S by a slightly weaker Condition S':

\noindent\textbf{Condition S'. } There exists some constant $r>0$, such that for every $\bar x = (\bar x', \bar x_{n+1})\in M$ with a horizontal unit outer normal (denote it by $\bar\nu = (\bar\nu', 0)$), the vertical cylinder $|x' - (\bar x'+r\bar\nu')|=r$ has an empty intersection with $G$. ($G$ is the bounded open set in $\mathbb{R}^{n+1}$ bounded by the hypersurface $M$.)

\begin{figure}[h!]
\includegraphics[scale=0.8]{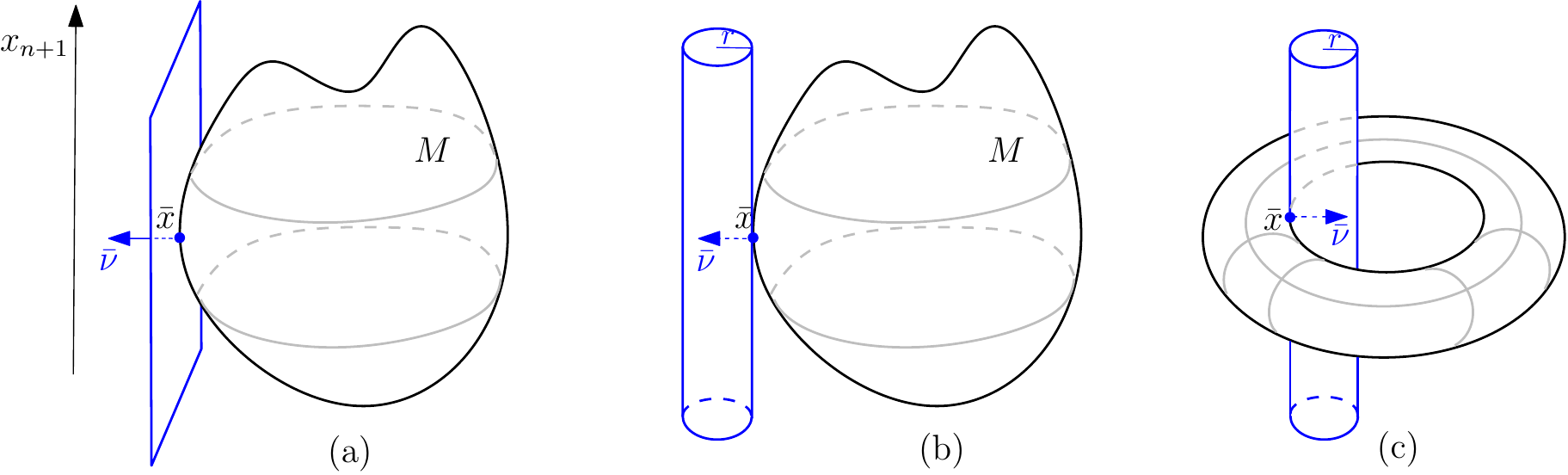}
\caption{(a) Illustration of Condition S. (b) Illustration of Condition S'. \\(c) The torus satisfies Condition~S', but not Condition S.\label{fig1}}
\end{figure}

See Figure~\ref{fig1} for an illustration of the  difference between Condition S and S'. Clearly, Condition~S' becomes more restrictive as the constant $r> 0$ increases. Note that in the $r\to+\infty$ limit, Condition S' becomes Condition S. 

The main theorem of this paper is as follows. 
\begin{thm}\label{thm1}
Let $M$ be a compact connected $C^2$ hypersurface without boundary embedded in $\mathbb{R}^{n+1}$, which satisfies both the Main Assumption and Condition S'. Then $M$ must be symmetric about some hyperplane $x_{n+1}=c$.
\end{thm}


Instead of the moving plane method, our proof has a variational flavor. More precisely, we will deform $M$ using a $C^2$ vector field $V: \mathbb{R}^{n+1}\to \mathbb{R}^{n+1}$, and consider the one-parameter family $\{M(t)\}_{t\in\mathbb{R}}$ of hypersurfaces
\begin{equation}
M(t) := \{x + t V(x): x \in M\}.
\end{equation}
Let $S(t) := \int_{M(t)} d\sigma$ be the surface area of $M(t)$. The key idea is to carefully choose some vector field $V$, then use two different ways to compute the first variation of the surface area at $t=0$ (i.e. computing $\frac{d}{dt}S(t)|_{t=0}$), and obtain a contradiction if $M$ is not symmetric about any $x_{n+1}=c$. 

In Section \ref{sec2}, we first establish some preliminary properties of the hypersurface $M$ when it satisfies the Condition S'. In particular, we will show that its projection $R$ on the hyperplane $x_{n+1}=0$ has a $C^{1,1}$ boundary, and each vertical line with $x'\in R^\circ$ intersects $M$ exactly at two points. We then prove in Section \ref{sec3} the symmetry of $M$ using a variational approach. We start with a warm-up result in Proposition~\ref{prop_eq}: as we ``deform'' $M$ using the constant vector field $V \equiv e_{n+1}$ and compute $\frac{d}{dt}S(t)|_{t=0}$ in two different ways, a short argument gives that the inequality $H(x', b) \leq H(x',a)$ for $a<b$ in the Main Assumption must actually be an equality for every pair of points. Building on this result, we finally present the proof of Theorem~\ref{thm1} using another carefully chosen vector field $V$.

\subsection*{Notations} 
For any $E\subset \mathbb{R}^{n+1}$, let $\pi(E)$ denote the projection of $E$ into the first $n$ coordinates, that is,
\[
\pi(E) :=  \{x'\in \mathbb{R}^n: (x', x_{n+1})\in E \text{ for some } x_{n+1}\}.
\] 
In particular, let $R := \pi(M)$ be the projection of $M$ on $\mathbb{R}^n$, which we will use extensively in this paper. The fact that $M$ is a compact connected closed hypersurface yields that $R\subset\mathbb{R}^n$ is bounded, closed, and connected. Throughout this paper we let $\partial R$ be the boundary of $R$ in $\mathbb{R}^n$.

In this proof we will work with balls in both $\mathbb{R}^n$ and $\mathbb{R}^{n+1}$. To avoid confusion, we denote by $B_r^{n+1}(x)$  the ball in $\mathbb{R}^{n+1}$ centered at $x$ with radius $r$, and  $B_r^n(x')$ the ball in $\mathbb{R}^{n}$ centered at $x'$ with radius $r$. 

For a  set $E \subset \mathbb{R}^d$ (where we will take either $d=n$ or $d=n+1$ in the proof), we say that its boundary $\partial E$ satisfies the \emph{interior ball condition} with radius $\rho$ if for every $x\in\partial E$, there is an open ball $B_x\subset E$ with radius $\rho$ such that $x\in \partial B_x$. Likewise, we say $\partial E$ satisfies the \emph{exterior ball condition} with radius $\rho$ if for every $x\in\partial E$, there is an open ball $B_x\subset E^c$ with radius $\rho$ such that $x\in \partial B_x$. Note that since $M = \partial G$ is a $C^2$ hypersurface embedded in $\mathbb{R}^{n+1}$, it satisfies both the interior and exterior ball condition with radius $\rho$ for some $\rho>0$.

\section{Preliminary properties of the hypersurface}\label{sec2}
 In the next proposition, we will establish some preliminary properties of the hypersurface $M$ when it satisfies the Condition S'.

\begin{prop}\label{prop_S}
Let $M$ be a compact connected $C^2$ hypersurface without boundary embedded in $\mathbb{R}^{n+1}$, which satisfies Condition S'. Then we have the following:
\begin{enumerate}[(a)]

\item For every $\bar x=(\bar x',\bar x_{n+1})\in M$, it has a horizontal outer normal if and only if $\bar x'\in \partial R$.

\item $\partial R$ satisfies both the interior and exterior ball condition with radius $\rho_0$ for some $\rho_0 \in (0,r]$ (here $r>0$ is the constant in Condition S'), and has $C^{1,1}$ regularity.

\item $M = M_1 \cup M_2 \cup \hat M$, where
\begin{equation}\label{mhat}
\hat M := \{(x', x_{n+1})\in M : x'\in\partial R\},
\end{equation}
and
 $M_1, M_2$ are graphs of two functions $f_1, f_2: R^\circ \to\mathbb{R}$, with $f_1, f_2 \in C^2(R^\circ)$, and $f_1>f_2$ in $R^\circ$. \end{enumerate}
\end{prop}

\noindent\textbf{Remark.} Note that one can construct examples where $M$ satisfies the assumptions of the proposition but $f_1, f_2$ are discontinuous in $R$ up to the boundary, therefore we can only conclude that $f_1, f_2 \in C^2(R^\circ)$ in (c). 
$\hat{M}$ is measurable, since $\hat{M}=M\setminus(M_1\cup M_2)$, and both $M_1$ and $M_2$ are measurable. 

\begin{proof}
The proof of (a) is rather straightforward. For any $\bar x=(\bar x',\bar x_{n+1})\in M$ with $\bar x' \in \partial R$,  the outer normal at $\bar x$ must be horizontal: if not, using the fact that $M$ is a $C^2$ hypersurface without boundary, we would have $\bar x' \in R^\circ$. The ``only if'' direction is a consequence of Condition~S'. Take any $\bar x\in M$ with a horizontal outer normal $(\bar \nu',0)$. Let  $U := \{(x',x_{n+1}): |x' - (\bar x'+r\bar\nu')| < r\}$ be the interior of the vertical cylinder given by Condition S', and note that $\bar x\in \partial U$. Condition~S' gives that $\partial U \cap G = \emptyset$, implying that $U \cap M = \emptyset$. This is equivalent to $\pi(U) \cap R = \emptyset$. Note that $\pi(U)$ is an open ball in $\mathbb{R}^n$ with $x'$ on the boundary, which implies $x' \in \partial R$.

Next we prove that $\partial R$ satisfies the exterior ball condition with radius $r$. For any $\bar x' \in \partial R$, using that $R$ is closed, there exists some $\bar x := (\bar x', \bar x_{n+1})\in M$. Denote the unit outer normal of $M$ at $\bar x$ by $\bar \nu := (\bar \nu', \bar \nu_{n+1})$. (See Figure~\ref{fig2} for an illustration.) Part (a) gives that $\bar \nu_{n+1}=0$. Let $U$ be given as in the paragraph above, and the same argument gives that $\pi(U) \subset R^c$, where $\pi(U)$ is an open ball in $\mathbb{R}^n$ with radius $r$, with $\bar x'$ on the boundary. Thus $R$ satisfies the exterior ball condition with radius $r$. 

To show the interior ball condition, take any $\bar x' \in \partial R$, and let $\bar x \in M$ be given as above. Since $M$ is a $C^2$ surface embedded in $\mathbb{R}^n$, there is some $\rho>0$ only depending on $M$, such that there exists an open ball $B_{\bar x} \subset G$ with radius $\rho$, which satisfies $\bar x \in \partial B_{\bar x}$. The ball must be tangent to $M$ at $\bar x$, thus can be written as $B^{n+1}_\rho(\bar x - \rho\bar\nu)$. Since $\bar\nu_{n+1}=0$ by (a), taking the projection $\pi$ yields that $B^n_\rho(\bar x'-\rho\bar\nu') \subset \pi(G)\subset R^\circ$, thus $\partial R$ satisfies the interior ball condition with radius $\rho>0$. Finally, setting $\rho_0 := \min\{r,\rho\}>0$ gives that $\partial R$ satisfies both the interior and exterior ball conditions with radius $\rho_0$, and it is well-known that this implies $\partial R \in C^{1,1}$ (see \cite[Lemma 2.2]{aikawa2007boundary} for a proof). 

\begin{figure}[h!]
\includegraphics[scale=0.8]{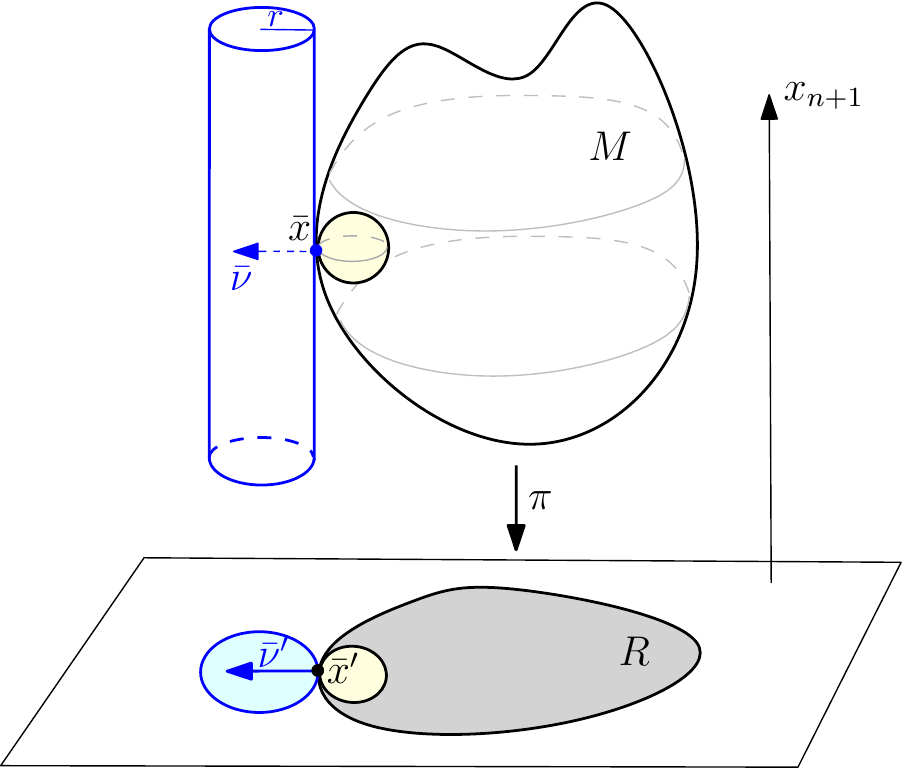}
\caption{Illustration of part (b) of the proof of Proposition~\ref{prop_S}.\label{fig2}}
\end{figure}

Next we move on to (c). Let us define $f_1:R^\circ \to \mathbb{R}$ as
\[
f_1(x') := \sup\{x_{n+1}: (x', x_{n+1}) \in M\} \quad\text{ for }x'\in R^\circ.
\]
Since $M$ is closed, and $\pi(M)=R$, we know that $f_1$ is well-defined for all $x'\in R^\circ$, and $(x', f_1(x'))\in M$ for any $x'\in R^\circ$.  Next we will show that $f_1 \in C(R^\circ)$. 

To show $f_1$ is upper semi-continuous at any $x_0' \in R^\circ$,  for any sequence of points $\{x_i'\}_{i=1}^\infty \subset R^\circ$ that converges to $x_0'$, we have $(x_i', f_1(x_i'))\in M$. This sequence has an accumulation point $(x_0', \limsup_{i\to\infty} f_1(x_i'))$, which is in $M$ since $M$ is closed. Thus by definition of $f_1$ we have $f_1(x_0') \geq \limsup_{i\to\infty} f_1(x_i')$. For the lower semi-continuity at $x_0' \in R^\circ$, by part (a), the outer normal at $(x_0', f_1(x_0'))\in M$ is not horizontal. Thus in a neighborhood of $x_0'$, $M$ can be locally parametrized as the graph of $(x', g(x'))$ for some $C^2$ function $g$, where $g(x_0')=f_1(x_0')$. The definition of $f_1$ yields that $f_1(x')\geq g(x')$ in this neighborhood, thus $f_1(x_0')=\lim_{x'\to x_0'} g(x') \leq \liminf_{x' \to x_0'} f_1(x')$. This finishes the proof that $f_1 \in C(R^\circ)$.

Note that 
\[M_1 := \{(x', f_1(x')): x'\in R^\circ\}\] is a subset of $M$, thus we have $f_1\in C^2(R^\circ)$ due to $M$ being $C^2$ and the fact that $M$ does not have horizontal outer normal in $\pi^{-1}(R^\circ)$. In addition, since $\pi(G)=R^\circ$ and $G$ is connected (which follows from that $M$ is connected), we have that $R^\circ$ is connected, and combining this with the continuity of $f_1$ yields that $M_1$ is connected. Let $M_{in} := M \cap \pi^{-1}(R^\circ)$. Note that $M_1$ is in fact a connected component of $M_{in}$ in view of (a). 

Now let us consider the set $M_{in} \setminus M_1$. Since $\pi(G)=R^\circ$, each vertical line with $x' \in R^\circ$ must intersect $M_{in}$ at least twice, implying that $\pi(M_{in}\setminus M_1)$ still covers the whole $R^\circ$. This allows us to define $f_2: R^\circ \to \mathbb{R}$ as
\[
f_2(x') := \sup
\{x_{n+1}: (x', x_{n+1}) \in M_{in} \setminus M_1\} \quad\text{ for }x'\in R^\circ.
\]
The same argument as $f_1$ also yields that $f_2 \in C^2(R^\circ)$, and $M_2 := \{(x', f_2(x')): x'\in R^\circ\}$ is another connected component of $M_{in}$. It is clear that $f_1 > f_2$ in $R^\circ$. Since $M_1, M_2 \subset M = \partial G$, and $M_2$ lies below $M_1$, 
we know that $G$ must be  between $M_1$ and $M_2$ (recall that $G$ is connected). Thus $M_{in} \subset \partial G$ cannot have any connected component 
below $M_2$. 
As a result, we have $M_{in} = M_1 \cup M_2$, i.e. $M = M_1 \cup M_2 \cup \hat M$, with $\hat M$ given by \eqref{mhat}.
\end{proof}

\section{Symmetry by a variational approach} \label{sec3}

Under Condition S', we have shown in Proposition~\ref{prop_S}(c) that $M$ can be partitioned into $M_1\cup M_2 \cup \hat M$, where $M_1, M_2$ are graphs of two functions $f_1, f_2 \in C^2(R^\circ)$, and $f_1>f_2$ in $R^\circ$. Due to the Main Assumption, we have the inequality 
\begin{equation}\label{ineq_assumption}
H(x', f_1(x')) \leq H(x', f_2(x'))\quad\text{ for all }x' \in R^\circ.
\end{equation} 
As a warm-up, let us first explain how to use a variational approach to prove a weaker result: namely, we will show that the inequality in \eqref{ineq_assumption} must actually be an equality for all $x'\in R^\circ$.

\begin{prop}\label{prop_eq}
Under the assumptions of Theorem~\ref{thm1}, for each $x' \in R^\circ$, the mean curvature of the two intersections must be identical, i.e. 
\[
H(x', f_1(x')) = H(x', f_2(x')) \quad\text{ for all } x' \in R^\circ,
\]
where $f_1, f_2$ are as given in Proposition~\ref{prop_S}(c).
\end{prop}

\begin{proof}

Let $V(x) = e_{n+1} = (0,\dots, 0, 1)$, and consider the family of set $M(t)$ given by
\begin{equation}\label{def_Mt}
M(t) := \{x + t V(x): x \in M\}.
\end{equation} Let $S(t) := \int_{M(t)} d\sigma$ be the surface area of $M(t)$. On the one hand, clearly $M(t)$ is a translation of $M$ upwards by $t$ units, thus the surface area $S(t)$ remains invariant for all $t\in\mathbb{R}$, implying 
\begin{equation}\label{ds1}
\frac{d}{dt} S(t)\Big|_{t=0} = 0.
\end{equation}

On the other hand, for any $C^2$ vector field $V(x): \mathbb{R}^{n+1}\to \mathbb{R}^{n+1} $ (which is indeed the case for our $V$ since it is a constant vector field), the first variation of surface area \cite[page 7]{CM} is given by 
\begin{equation}\label{first_var}
\begin{split}
\frac{d}{dt} S(t)\Big|_{t=0} &= -\int_{M} V(x) \cdot \nu(x) H(x) \,d\sigma(x),
\end{split}
\end{equation}
where $\nu(x)$ is the unit outer normal at $x$ for $x\in M$. In the rest of this proof, we aim to show that the right hand side is strictly positive if we have a strict inequality in \eqref{ineq_assumption} for some $x' \in R^\circ$, leading to a contradiction with \eqref{ds1}.

To see this, we break the right hand side of \eqref{first_var} into the integrals on $M_1$, $M_2$ and $\hat M$, and recall that $V = e_{n+1}$.  Proposition~\ref{prop_S}(a) yields that $V(x) \cdot \nu(x)=\nu_{n+1}(x)=0$ on $\hat M$, thus the integral on $\hat M$ is zero. By Proposition~\ref{prop_S}(c), $M_i$ is the graph of $(x', f_i(x'))$ for $i=1,2$ and $x'\in R^\circ$, where $f_1 > f_2$, leading to 
\begin{equation}\label{ineq_S}
\begin{split}
\frac{d}{dt} S(t)\Big|_{t=0} &= \sum_{i=1}^2 - \int_{M_i} e_{n+1}\cdot \nu(x) H(x)\, d\sigma(x) \\
&=- \int_{R^\circ} (H(x', f_1(x')) - H(x', f_2(x')))\,dx'\\
&\geq 0.
\end{split}
\end{equation}
Here in the second equality we used that $|e_{n+1}\cdot \nu(x)| d\sigma(x) = dx'$, as well as the fact that $e_{n+1}\cdot \nu(x)$ is positive for $x\in M_1$, and negative for $x\in M_2$. The last inequality comes from the assumption \eqref{ineq_assumption}.
 
Note that $M$ being a $C^2$ hypersurface implies $H(x', f_i(x'))$ is continuous in $R^\circ$ for $i=1,2$. 
Thus if we have a strict inequality in \eqref{ineq_assumption} for some $x'_0\in R^\circ$, it implies $H(x', f_1(x')) - H(x', f_2(x'))< 0$ in some open neighborhood of $x'_0$, leading to a strict inequality in \eqref{ineq_S}, thus contradicting \eqref{ds1}. As a result, the inequality in \eqref{ineq_assumption} must be an equality for all $x'\in R^\circ$.
\end{proof}

Now we are ready to prove the main theorem. 

\begin{proof}[Proof of Theorem \ref{thm1}] Our goal is to show that $f_1 + f_2 \equiv c_0$ in $R^\circ$ for some constant $c_0$, which immediately implies that $M$ is symmetric about the hyperplane $x_{n+1} = \frac{c_0}{2}$.

Towards a contradiction, assume that $f_1 + f_2$ is not a constant in $R^\circ$. We will deform $M$ using a vector field $V$ that is a vertical shear flow, i.e. 
\[
V(x) = V(x',x_{n+1}) = (0,\dots, 0, v(x')) = v(x') e_{n+1}\quad\text{ for } x\in\mathbb{R}^{n+1},
\] where $v \in C^{2}(\mathbb{R}^n)$ will be fixed later. We again compute $\frac{d}{dt}S(t)|_{t=0}$ in two ways.

On the one hand, since $V(x) = v(x') e_{n+1}$ is a $C^2$ vector field in $\mathbb{R}^{n+1}$, the first variation of surface area \cite[page 7]{CM} and a similar argument to \eqref{ineq_S} again give
\[
\begin{split}
\frac{d}{dt} S(t)\Big|_{t=0} &= -\int_{M} V(x) \cdot \nu(x) H(x) \,d\sigma(x)\\
&=-\sum_{i=1}^2\int_{M_i} v(x') e_{n+1}\cdot \nu(x) H(x) \,d\sigma(x)\\
&=- \int_{R^\circ} v(x')  \big(H(x', f_1(x')) - H(x', f_2(x'))\big)\,dx',
\end{split}
\]
where $\nu(x)$ is the unit outer normal at $x$ for $x\in M$, and in the second equality we use the fact that $e_{n+1}\cdot \nu(x)\equiv 0$ for all $x\in\hat M$. By Proposition~\ref{prop_eq}, the integrand on the right hand side is zero for all $x' \in R^\circ$, leading to 
\begin{equation}\label{ds11}
\frac{d}{dt} S(t)\Big|_{t=0} = 0.
\end{equation}

On the other hand, if $f_1+f_2$ is not a constant in $R$, we will construct a $v \in C^2(\mathbb{R}^n)$ such that $\frac{d}{dt}S(t)|_{t=0} > 0$.  
Heuristically, we will define $v=f_1+f_2$ in most of $R^\circ$, and smoothly cut it off to zero near $\partial R$ as follows. For a sufficiently small $\delta>0$ that we will fix later, let
 \[
 R_\delta := \{x' \in R: \text{dist}(x', \partial R)>\delta\}.
 \]  
Let $\eta \in C^\infty(\mathbb{R}^n)$ be a standard mollifier supported in the unit ball, with $\eta \geq 0, \int_{\mathbb{R}^n} \eta(x') dx' = 1$ and $|\nabla \eta|\leq C(n)$. For any $a>0$, denote by $\eta_{a}(x') := a^{-n} \eta(a^{-1} x')$ its dilation. For $x'\in\mathbb{R}^n$, let
\[
\phi_\delta(x') :=  (1_{R_{2\delta/3}}*\eta_{\delta/3})(x')
\]
be a ``smooth cut-off function''. Clearly, $\phi_\delta \in C^\infty(\mathbb{R}^n)$ is nonnegative, and satisfies $\phi_\delta \equiv 1$ in $R_\delta$, $\phi_\delta \equiv 0$ in $\mathbb{R}^n \setminus R_{\delta/3}$.
 In addition, Young's inequality for convolution gives
\begin{equation}\label{gradient_phi}
\sup_{\mathbb{R}^n}|\nabla\phi_\delta| \leq \|1_{R_{2\delta/3}}\|_{L^\infty(\mathbb{R}^n)} \|\nabla \eta_{\delta/3}\|_{L^1(\mathbb{R}^n)} = \frac{3}{\delta} \|\nabla\eta\|_{L^1(\mathbb{R}^n)} \leq \frac{C(n)}{\delta}.
\end{equation}
We now define $v:\mathbb{R}^n \to \mathbb{R}$ as
 \begin{equation}\label{v_inside}
 v(x') = \begin{cases} (f_1(x')+f_2(x')) \phi_\delta(x')& \text{ for }x'\in R^\circ\\
 0 & \text{ for }x'\in \overline{R^c}.
 \end{cases}
 \end{equation}
 Note that such definition indeed leads to $v\in C^2(\mathbb{R}^n)$: the smoothness of $\phi_\delta$ and the fact that $f_1+f_2 \in C^2(R^\circ)$ yield that $v\in C^2(R^\circ)$, and combining this with the fact that  $v\equiv 0$ in $\mathbb{R}^{n}\setminus R_{\delta/3}$ gives that $v\in C^2(\mathbb{R}^n)$. In addition, we have $v = f_1+f_2$ in $R_{\delta}$.

For $i=1,2$, let $M_i(t)$ be the surface $\{(x', x_{n+1}): x'\in R^\circ, x_{n+1}=f_i(x')+v(x')t\}$. Recall that $f_i \in C^2(R^\circ)$ for $i=1,2$ by Proposition~\ref{prop_S}(c). Since $v \in C^2(\mathbb{R}^n)$, it follows that the map $x' \mapsto f_i(x')+v(x')t$ is in $C^2(R^\circ)$ for any $t\in\mathbb{R}$. 
Since $M(t)=M_1(t)\cup M_2(t) \cup \hat M$ (here $\hat M$ remains unchanged in $t$ since $v\equiv 0$ in a small neighborhood of $\partial R$), its surface area at a given $t$ can be computed as
  \[
     S(t)=\sum_{i=1}^{2}\int_{R^\circ}\sqrt{1+|\nabla (f_i(x')+v(x')t)|^2}\,dx' + \hat S,    \]
where $\hat S$ is the surface area of $\hat M$.  Note that $S(t)$ is differentiable in $t$ since $v$ is supported in $R_{\delta/3}$, and $\|f_i\|_{C^2(R_{\delta/3})}$ is finite for $i=1,2$.
Taking its derivative in $t$ and  setting $t=0$ yields 
  \begin{equation}\label{def_I}
     I:= \frac{dS(t)}{dt}\Big|_{t=0}=\int_{R_{\delta/3}}\sum_{i=1}^{2} \frac{\nabla f_i(x')\cdot \nabla v(x')}{\sqrt{1+|\nabla f_i(x')|^2}} \,dx',
  \end{equation}
 where we use that $v\equiv 0$ in $\mathbb{R}^n \setminus R_{\delta/3}$.
Note that the  integral in \eqref{def_I} is convergent since $\sup_{\mathbb{R}^n} |\nabla v|<\infty$ and $\frac{|\nabla f_i(x')|}{\sqrt{1+|\nabla f_i(x')|^2}}< 1$ in $R_{\delta/3}$.

With $v$ defined by \eqref{v_inside}, we have
\[
\nabla v = \phi_\delta  \nabla(f_1+f_2) + (f_1+f_2) \nabla \phi_\delta \quad\text{ in }R_{\delta/3}.
\]
Plugging the above into \eqref{def_I}, we can decompose $I$ into $I_\delta^1+I_\delta^2$ as follows (where we use that $\text{supp}|\nabla \phi_\delta| \subset R_{\delta/3} \setminus R_{\delta}$):
\begin{align}
I &= \int_{R_{\delta/3}}\underbrace{\sum_{i=1}^{2} \frac{\nabla f_i(x')\cdot \nabla (f_1+f_2)(x')}{\sqrt{1+|\nabla f_i(x')|^2}}}_{=:F(x')} \phi_{\delta}(x') dx' +\displaystyle \int_{R_{\delta/3} \setminus R_\delta }\sum_{i=1}^{2} \frac{\nabla f_i(x')\cdot \nabla \phi_\delta(x') (f_1+f_2)(x')}{\sqrt{1+|\nabla f_i(x')|^2}} dx' \nonumber\\
& =: I_{\delta}^1+ I_{\delta}^2.\label{def_F}
\end{align} 

We will show the following property for $I_\delta^1$.

\noindent\textbf{Claim 1.} If $f_1+f_2 \neq \text{const}$ in $R^\circ$, then there exists some $a_0>0$, such that $I_\delta^1 \geq a_0>0$ for all sufficiently small $\delta>0$.

\noindent\emph{Proof of Claim 1}: For any $q\in\mathbb{R}^n$, define  $A(q):=\sqrt{1+|q|^2}$. Then $\nabla A(q)=\frac{q}{\sqrt{1+|q|^2}}$, and 
\[
   \partial^2_{q_iq_j}A(q)=(1+|q|^2)^{-\frac{3}{2}}(\delta_{ij}+|q|^2\delta_{ij}-q_iq_j) \quad\text{for } 1\le i,j\le n,
\]
where $\delta_{ij}=1$ if $i=j$, and $0$ if $i\neq j$.
So the Hessian of $A$ satisfies
\begin{equation}\label{eqA_H}
   \nabla^2A(q)\ge (1+|q|^2)^{-\frac{3}{2}}I>0 \quad\text{ for all }q\in \mathbb{R}^n,
\end{equation}
where the two inequalities are in the following sense: we say two $n\times n$ symmetric matrices $U,V$ satisfies $U> V$ (or $U\geq V$) if $U-V$ is positive definite (or  positive semi-definite). Note that \eqref{eqA_H} implies that $A$ is strict convex in $\mathbb{R}^n$.

By the definition of $F(x')$ in \eqref{def_F}, we have 
\begin{equation*}
   F(x')=\sum_{i=1}^{2}\nabla A(\nabla f_i(x'))\cdot (\nabla f_1+\nabla f_2).
\end{equation*}
Denoting $q_1:=\nabla f_1$ and $q_2:=-\nabla f_2$, we rewrite the above equation as
\[
   F(x')=\big(\nabla A(q_1(x'))-\nabla A(q_2(x'))\big) \cdot (q_1(x')-q_2(x')).
\]

For any fixed $x'$, applying the mean-value theorem to the scalar-valued function $g(t):=\nabla A(tq_1+(1-t)q_2)\cdot(q_1-q_2)$ for $0\leq t\leq 1$, we know there exists some $c\in[0,1]$ depending on $x'$, such that
\begin{equation}\label{F1}
\begin{split}
F(x')&= g(1)-g(0)=g'(c)  \\
&= (q_1(x')-q_2(x'))^{T}\, \nabla^2A\big(cq_1(x')+(1-c)q_2(x')\big) \,(q_1(x')-q_2(x'))\\
&\geq \big(1+(|q_1(x')|+|q_2(x')|)^2\big)^{-\frac{3}{2}}|q_1(x')-q_2(x')|^2,
\end{split}
\end{equation}
where we use \eqref{eqA_H} in the last inequality.
Therefore $F(x')\ge 0$ in $R^\circ$, and the equality holds if and only if $q_1(x')=q_2(x')$, i.e. $\nabla (f_1+f_2)(x')=0$.   

If $f_1+f_2\neq const$ in $R^\circ$, then since $f_1,f_2 \in C^2(R^{\circ})$,  there exists some $\bar{x}\in R^\circ$ and $\epsilon, b_1, b_2>0$, such that $B_{\epsilon}^n(\bar{x})\subset R_\delta$ for all sufficiently small $\delta>0$, and
\begin{equation}\label{eqf_12}
   |\nabla (f_1+f_2)(x')|\ge b_1 \textrm{~~ and ~~} |\nabla f_1(x')|+|\nabla f_2(x')|\le b_2\quad\text{ for } x'\in B_{\epsilon}^n(\bar{x}).
\end{equation}
By \eqref{F1} and \eqref{eqf_12}, for any $x'\in B_{\epsilon}^n(\bar{x})$, we have
\[
  \begin{split}
   F(x') 
           & \ge (1+b_2^2)^{-\frac{3}{2}}|q_1(x')-q_2(x')|^2\\
           & = (1+b_2^2)^{-\frac{3}{2}}|\nabla (f_1+ f_2)(x')|^2 \\
           &\geq (1+b_2^2)^{-\frac{3}{2}}b_1^2.
   \end{split}
\]

Combining this with the fact that $F(x')\ge 0$ in $R^\circ$, we obtain a lower bound of $I^1_{\delta}$ as follows, where we use that $\phi_\delta\equiv 1$ in $B_{\epsilon}^n(\bar{x})\subset R_\delta$, as well as $\phi_\delta\geq 0$:
\[
   \begin{split}
      I^1_{\delta} & \ge \int_{B_{\epsilon}^n(\bar{x})}F(x')dx'  \ge (1+b_2^2)^{-\frac{3}{2}}b_1^2|B_{\epsilon}^n(\bar{x})|
                         =:a_0>0.
   \end{split}
\]
finishing the proof of Claim 1.

In the rest of the proof, we aim to show that $| I_\delta^2| $ can be made arbitrarily small by setting $\delta$ small. Clearly one can bound it as
\[
|I_\delta^2| \leq \underbrace{|R_{\delta/3}\setminus R_\delta|}_{=:T_1} \, \underbrace{\sup_{R_{\delta/3}\setminus R_\delta} |(f_1+f_2)\nabla \phi_\delta|}_{=:T_2}\, \underbrace{\sup_{R_{\delta/3} \setminus R_\delta} \Big|\sum_{i=1}^{2} \frac{\nabla f_i}{\sqrt{1+|\nabla f_i|^2}}\Big|}_{=:T_3}
\]
where $|R\setminus R_\delta|$ denotes the Lebesgue measure of $R\setminus R_\delta$ in $\mathbb{R}^n$.

$T_1$ and $T_2$ are rather straightforward to control. Since $R$ is bounded and has a $C^{1,1}$ boundary by Proposition~\ref{prop_S}(b), there exists some $C_1(M,n)>0$ such that
  \begin{equation}\label{area_bd}
T_1 \leq |R \setminus R_\delta| \leq C_1(M,n) \delta
  \end{equation} 
  for all $\delta\in(0,1)$.
To bound $T_2$, by the definition of $\phi_\delta$ and the fact that $M$ is bounded (thus so are $f_1, f_2$),  we have 
 \[T_2 \leq \sup_{R^\circ}|f_1+f_2| \sup_{\mathbb{R}^n}|\nabla \phi_\delta| \leq C_2(M,n) \delta^{-1},\] where we use \eqref{gradient_phi} in the last inequality.

It is more delicate to bound the last term $T_3$. Note that the product $T_1 T_2$ is of order $O(1)$, thus the crude bound $T_3\leq 2$ is not sufficient. We claim that $T_3$ is actually of order $\sqrt{\delta}$, since the two terms in the sum has some nice cancellation properties:

\noindent\textbf{Claim 2.} Since $M$ is a $C^2$ hypersurface embedded in $\mathbb{R}^n$, it satisfies the interior ball property with radius $\rho>0$. Then for all $\delta>0$, we have
\begin{equation}\label{sum2}T_3:=\left|\sum_{i=1}^{2} \frac{\nabla f_i}{\sqrt{1+|\nabla f_i|^2}}\right| \leq 2\sqrt{\frac{2(\rho+r)}{\rho r}\delta}\quad\text{ in }R^\circ\setminus R_\delta,
\end{equation}
where $r>0$ is the constant in Condition S'.

\noindent\emph{Proof of Claim 2}. 
Take any $\bar x' \in R^\circ\setminus R_\delta$. For $i=1,2$, let $\nu_i = (\nu_i', \nu_i^{n+1})$ be the unit outer normal of $M$ at the point $\bar{x}_i := (\bar x',f_i(\bar x'))$. By Proposition~\ref{prop_S}(a), $\nu_i^{n+1}\neq 0$. We then have 
\[
   \nu_1=(\nu_1', \nu_1^{n+1})=\frac{(-\nabla f_1(\bar{x}'), 1)}{\sqrt{1+|\nabla f_1(\bar{x}')|^2}}, \quad 
   \nu_2=(\nu_2', \nu_2^{n+1})=\frac{(\nabla f_2(\bar{x}'), -1)}{\sqrt{1+|\nabla f_2(\bar{x}')|^2}}.  
\]
Hence, one has 
 \begin{equation}\label{sum3}
   T_3=\left|\frac{\nabla f_1(\bar{x}')}{\sqrt{1+|\nabla f_1(\bar{x}')|^2}}+\frac{\nabla f_2(\bar{x}')}{\sqrt{1+|\nabla f_2(\bar{x}')|^2}}\right|=|\nu_1'-\nu_2'|, 
 \end{equation}
%
%
%
 so it suffices to bound the right hand side. 
 

 Since $\bar x'\in R^\circ\setminus R_\delta$, there exists some $x_0' \in \partial R$, such that $|\bar x'-x_0'| \leq \delta$. Using that $R$ is closed, there exists some $x_0 := (x_0', x_0^{n+1}) \in M$ that projects to $x_0'$. (If there are more than one such points, let $x_0$ be any of them.) By Proposition~\ref{prop_S}(a), $M$ has a horizontal outer normal at $x_0$, which we denote by $\nu_0 = (\nu_0', 0)$.  See Figure~\ref{fig2} for an illustration of the points.
 
 \begin{figure}[h!]
\includegraphics[scale=0.8]{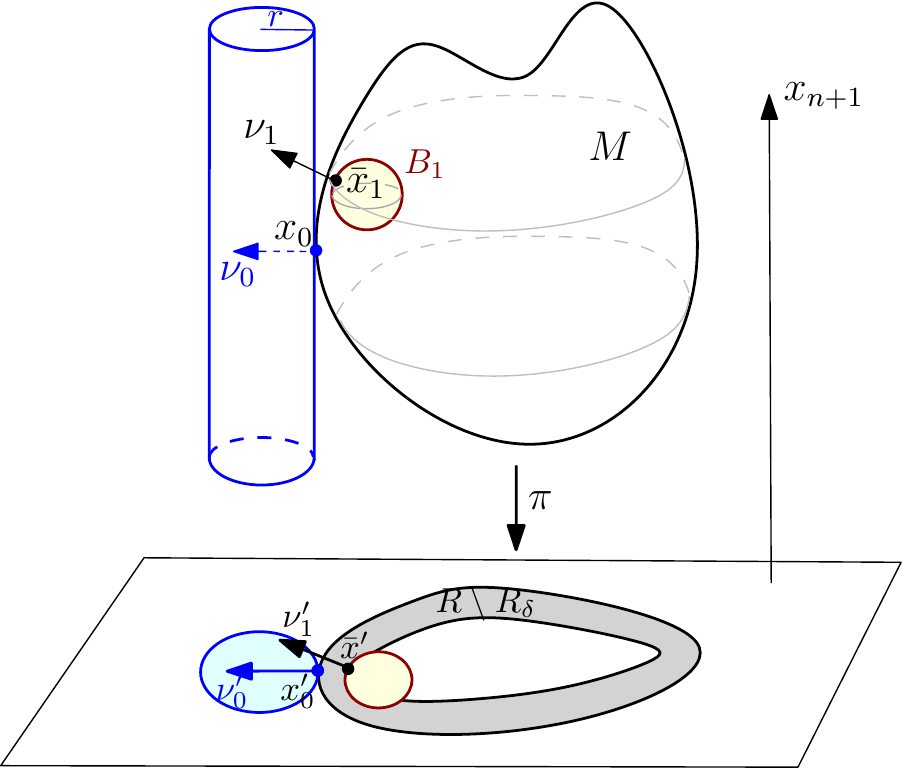}
\caption{Illustration of the proof of Claim 2.\label{fig3}}
\end{figure}

  By Condition S', the cylinder $|x' - (x_0'+r\nu_0')|=r$ has an empty intersection with $G$. On the other hand,  $M$ satisfies the interior ball condition with radius $\rho>0$, thus the balls $B_i := B^{n+1}_{\rho}(\bar x_i - \rho \nu_i)$ satisfy $B_i \subset G$ for $i=1,2$. As a result, the open ball $B_i$ must completely lie outside the cylinder, implying that its center must have distance at least $\rho+r$ to the axis of cylinder. That is,
  \begin{equation}\label{diff_center}
  |(\bar x' - \rho \nu_i' ) - (x_0' + r\nu_0')| \geq \rho+r\quad\text{ for } i=1,2,
  \end{equation}
where we used that $\bar x_1' = \bar x_2' = \bar x'$.
  Since $|\bar x'-x_0'|<\delta$ (this follows from our choice of $x_0$), \eqref{diff_center} implies that
  \[
  |\rho \nu_i' + r\nu_0'| \geq \rho + r -\delta.
  \]
Taking square of both sides and using the facts that $|\nu_0'|=1$ and  $|\nu_i'|< 1$, for all $\delta>0$ we have  
\[
 \rho^2\underbrace{(1-|\nu_i'|^2)}_{> 0} + r^2\underbrace{(1-|\nu_0'|^2)}_{= 0}+ 2\rho r \underbrace{(1-\nu_i' \cdot \nu_0')}_{> 0} \leq 2 (\rho+r)\delta-\delta^2 < 2(\rho+r)\delta.
\]
This directly leads to
\[
\quad 1-\nu_i' \cdot \nu_0' < \frac{\rho+r}{\rho r}\delta \quad\text{ for }i=1,2,
\]
allowing us to bound $|\nu_i' - \nu_0'|$ as follows (where again we used that $|\nu_0'|=1$ and $|\nu_i'|< 1$:
\[
|\nu_i' - \nu_0'|^2 = |\nu_i'|^2+1 - 2 \nu_i'\cdot \nu_0' < 2 - 2 \nu_i'\cdot \nu_0' < \frac{2(\rho+r)}{\rho r}\delta\quad\text{ for }i=1,2.
\]
As a result, we have \[
|\nu_2' - \nu_1'| \leq |\nu_1'-\nu_0'| + |\nu_2'-\nu_0'| \leq 2\sqrt{\frac{2(\rho+r)}{\rho r}\delta}.
\]
Plugging this into \eqref{sum3} finishes the proof of Claim 2.

Once we prove Claim 2, the bounds on $T_1,T_2, T_3$ yield that $|I_\delta^2| \leq C(M,n,r)\sqrt{\delta}$ for all $\delta>0$, thus by setting $\delta\in(0,1)$ sufficiently small and using Claim 1, we have that $I > a_0/2>0$. This contradicts with \eqref{ds11}, thus the proof is finished.
  \end{proof}

\FloatBarrier

\end{document}